\documentclass[final,3p, times, 12pt]{elsarticle}

\usepackage[centertags]{amsmath}
\usepackage{amsfonts}
\usepackage{amssymb}
\usepackage{amsthm}

\usepackage{graphicx}
\usepackage{multicol}

\theoremstyle{plain}
\newtheorem{thm}{Theorem}[section]
\newtheorem{cor}[thm]{Corollary}
\newtheorem{lem}[thm]{Lemma}
\newtheorem{prop}[thm]{Proposition}

\theoremstyle{definition}
\newtheorem{defn}[thm]{Definition}
\newtheorem{exam}[thm]{Example}
\newtheorem{rem}[thm]{Remark}

\theoremstyle{remark}
\numberwithin{equation}{section}

\newproof{pf}{Proof}

\begin{document}

\begin{frontmatter}

\title{Spectral characterizations of two families \\ of nearly complete bipartite graphs}

\author[CL]{Chia-an Liu\corref{cor}}\ead{liuchiaan8@gmail.com}
\author[CW]{Chih-wen Weng}\ead{weng@math.nctu.edu.tw}

\date{January 26, 2016}
\cortext[cor]{Corresponding author}

\address[CL]{Department of Financial and Computational Mathematics, I-Shou University, Kaohsiung, Taiwan}
\address[CW]{Department of Applied Mathematics, National Chiao-Tung University, Hsinchu, Taiwan}

\begin{abstract}
It is not hard to find many complete bipartite graphs which are not determined by their spectra. We show that the graph obtained by deleting an edge from a complete bipartite graph is determined by its spectrum. We provide some graphs, each of which is obtained from a complete bipartite graph by adding a vertex and an edge incident on the new vertex and an original vertex, which are not determined by their spectra.
\end{abstract}

\end{frontmatter}


\section{Introduction}   \label{s1}

The {\it adjacency matrix} $A=(a_{ij})$ of a simple graph $G$ is a $0$-$1$ square matrix with rows and columns indexed by the vertex set $VG$ of $G$ such that for any $i, j\in VG$, $a_{ij}=1$ iff $i,j$ are adjacent in $G.$ The {\it spectrum} of $G$ is the set of eigenvalues of its adjacency matrix $A$ together with their multiplicities. Two graphs are {\it cospectral} (also known as isospectral) if they share the same graph spectrum. To start our study, let us consider the smallest non-isomorphic cospectral graphs first given by Cvetkovi\'{c} \cite{c:71} as shown in Figure 1: the graph union $K_{2,2}\cup K_1$ and the star graph $K_{1,4},$ where $K_{p,q}$ denotes the complete bipartite graph of bipartition orders $p$ and $q.$ It is quick to check that their spectrum are both $\{[0]^3,\pm 2\}.$ More constructions of cospectral graphs can be found in \cite{gm:82,cdgt:88,hs:95,s:00,dh:02}.

\begin{center}
\begin{multicols}{2}
\begin{picture}(50,60)
\put(10,20){\circle*{3}} \put(10,40){\circle*{3}}
\put(40,10){\circle*{3}} \put(40,30){\circle*{3}}
\put(40,50){\circle*{3}}
\qbezier(10,20)(25,25)(40,30) \qbezier(10,20)(25,35)(40,50)
\qbezier(10,40)(25,35)(40,30) \qbezier(10,40)(25,45)(40,50)
\put(0, -30){$K_{2,2}\cup K_1$}
\end{picture}

\begin{picture}(50,80)
\put(10,60){\circle*{3}}
\put(40,10){\circle*{3}} \put(40,30){\circle*{3}} \put(40,50){\circle*{3}} \put(40,70){\circle*{3}}
\qbezier(10,60)(25,65)(40,70) \qbezier(10,60)(25,55)(40,50)
\qbezier(10,60)(25,45)(40,30) \qbezier(10,60)(25,35)(40,10)
\put(10, -10){$K_{1,4}$}
\end{picture}
\end{multicols}
\bigskip

{\bf Figure 1:} Two non-isomorphic cospectral graphs $K_{2,2}\cup K_1$ and $K_{1,4}.$
\end{center}

A graph $G$ is {\it determined by the spectrum} if all the cospectral graphs of $G$ are isomorphic to $G.$ We abbreviate `determined by the spectrum' to DS in the following. The question `which graphs are DS?' goes back for more than half a century and originates from chemistry~\cite{gp:56,cs:57}, \cite[Chapter~6]{cds:80}. After that, there appeared many examples and applications for the DS graphs. One of them is that in 1966 Fisher \cite{f:66} modeled the shape of a drum by a graph from considering a question of Kac \cite{k:66}: `Can you hear the shape of a drum?' To see more details, Van Dam, Haemers and Brouwer gave a great amount of surveys for the DS graphs \cite{dh:02,dh:03,dh:09},\cite[Chapter~14]{bh:11} in the past decades.

In~\cite{ds:04} the non-regular bipartite graphs with four distinct eigenvalues were studied, and whether a such connected graph on at most $60$ vertices is DS or not was determined. In~\cite{ds:05} bipartite biregular graphs with $5$ eigenvalues were studied, and all such connected graphs on at most $33$ vertices were determined. In this research we study two families of nearly complete bipartite graphs one-edge different from a complete bipartite graph which also have $4$ or $5$ distinct eigenvalues without the assumptions of regularity, connectivity, or bounds on the number of vertices.

\medskip

Let $G$ be a simple bipartite graph with $e$ edges. The {\it spectral radius} $\rho(G)$ of $G$ is the largest eigenvalue of the adjacency matrix of $G.$ It was shown in~\cite[Proposition~2.1]{bfp:08} that $\rho(G) \leq \sqrt{e}$ with equality if and only if $G$ is a complete bipartite graph with possibly some isolated vertices. It is direct that for any positive integer $p$ the regular complete bipartite graph $K_{p,p}$ is DS but, for example, the non-isomorphic bipartite graphs $K_{1,6}$ and $K_{2,3} \cup 2K_1$ are cospectral. There are several extending results~\cite{bfp:08,cfksw:10,lw:15,cfw:15} of the above bound, which aim to solve an analog of the Brualdi-Hoffman conjecture for non-bipartite graphs~\cite{bh:85}, proposed in~\cite{bfp:08}.

\medskip

Our research is motivated from the following {\it twin primes bound}  proposed in~\cite[Theorem~5.2]{cfw:15}: For $e \geq 4,$ $(e-1,e+1)$ is a pair of twin primes if and only if
$$\rho(e) < \sqrt{\frac{e+\sqrt{e^2-4(e-1-\sqrt{e-1})}}{2}}$$
where $\rho(e)$ denotes the maximal $\rho(G)$ of a bipartite graph $G$ on $e$ edges which is not a union of a complete bipartite graph and some isolated vertices.  We need to introduce the notations $K_{p,q}^-$ and $K_{p,q}^+$ of the graphs which are one-edge different from a complete bipartite graph. For $2 \leq \min\{p,q\},$ let $K_{p,q}^-$ denote the graph with $pq-1$ edges obtained from $K_{p,q}$ by deleting an edge, and $K_{p,q}^+$ denote the graph with $pq+1$ edges obtained from $K_{p,q}$ by adding a new vertex $x$ and a new edge $xy$ where $y$ is a vertex in the partite set of order $\min\{p,q\}.$ Note that $K_{2,q}^+ = K_{2,q+1}^-$ for $q \geq 2.$ Two examples of such graphs are shown in Figure 2.

\newpage

\begin{center}
\begin{multicols}{2}
\begin{picture}(50,100)
\put(10,40){\circle*{3}} \put(10,60){\circle*{3}}
\put(40,30){\circle*{3}} \put(40,50){\circle*{3}}
\put(40,70){\circle*{3}}
\put(10,20){\circle*{3}} \put(40,10){\circle*{3}}
\qbezier(10,60)(25,65)(40,70) \qbezier(10,60)(25,55)(40,50)
\qbezier(10,60)(25,45)(40,30) \qbezier(10,40)(25,55)(40,70)
\qbezier(10,40)(25,45)(40,50)
\qbezier(10,40)(25,35)(40,30)
\qbezier(10,20)(25,45)(40,70) \qbezier(10,20)(25,35)(40,50)
\qbezier(10,20)(25,25)(40,30)
\qbezier(10,60)(25,35)(40,10) \qbezier(10,40)(25,25)(40,10)
\put(10, -10){$K_{3,4}^-$}
\end{picture}

\begin{picture}(50,100)
\put(10,40){\circle*{3}} \put(10,60){\circle*{3}}
\put(40,30){\circle*{3}} \put(40,50){\circle*{3}}
\put(40,70){\circle*{3}}
\put(10,20){\circle*{3}} \put(40,10){\circle*{3}}
\put(40,90){\circle*{3}}
\qbezier(10,60)(25,65)(40,70) \qbezier(10,60)(25,55)(40,50)
\qbezier(10,60)(25,45)(40,30) \qbezier(10,40)(25,55)(40,70)
\qbezier(10,40)(25,45)(40,50)
\qbezier(10,40)(25,35)(40,30)
\qbezier(10,20)(25,45)(40,70) \qbezier(10,20)(25,35)(40,50)
\qbezier(10,20)(25,25)(40,30)
\qbezier(10,60)(25,35)(40,10) \qbezier(10,40)(25,25)(40,10)
\qbezier(10,60)(25,75)(40,90)
\qbezier(10,20)(25,15)(40,10)
\put(10, -10){$K_{3, 4}^+$}
\end{picture}
\end{multicols}
\bigskip

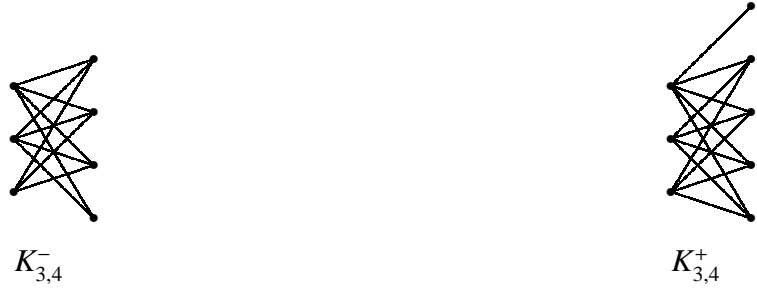
{\bf Figure 2:} The graphs $K_{3,4}^-$ and $K_{3,4}^+$ which are one-edge different from $K_{3,4}.$
\end{center}

The paper is organized as follows. Preliminary contents are in Section~\ref{sec_pre}. Theorem~\ref{thm_Kpq-} in Section~\ref{sec_Kpq-} proves that the all graphs $K_{p,q}^-$ for $2 \leq p \leq q$ are DS. Then Theorem~\ref{thm_Kpq+} in Section~\ref{sec_Kpq+} find the all pairs $(p,q)$ such that the bipartite graph $K_{p,q}^+$ is DS. Furthermore, for each $K_{p,q}^+$'s which is not DS we also find its unique non-isomorphic cospectral graph.

\bigskip

\section{Preliminary}   \label{sec_pre}

Basic results are provided in this section for later used.
\medskip

\begin{lem}(\cite[Proposition~2.1]{bfp:08})   \label{lem_sqrte}
Let $G$ be a simple bipartite graph  with $e$ edges. Then
$$\rho(G) \leq \sqrt{e}$$ with equality
iff $G$ is a disjoint union of a complete bipartite graph and isolated vertices.
\end{lem}

\medskip

The following result gives the relations between the spectrum and the numbers of vertices and edges in a graph which is proved simply by the definition of the adjacency matrix and its square.
\begin{prop}\label{prop_order_size}
Let $G$ be a simple graph with eigenvalues $\lambda_1 \geq \lambda_2 \geq \cdots \geq \lambda_n.$ Then
\begin{itemize}
\item[(i)] $G$ has $n$ vertices, and
\item[(ii)] $G$ has $\frac{1}{2}\sum\limits_{i=1}^n \lambda_i^2$ edges.
\end{itemize}
\end{prop}

\medskip

Since we focus on the bipartite graphs, a well-known spectral characterization of bipartite graphs \cite[Proposition~3.4.1]{bh:11} is used in this research.
\begin{prop}\label{prop_bipartite}
A simple graph $G$ is bipartite if and only if for each eigenvalue $\lambda$ of $G,$ $-\lambda$ is also an eigenvalue of $G$ with the same multiplicity.
\end{prop}

\medskip

Then a spectral characterization of complete bipartite graphs is direct.
\begin{prop}\label{prop_complete_bi}
Let $G$ be a simple graph with spectrum $\{[0]^{n-2},\pm \lambda\}$ where $n \geq 2$ is the number of vertices in $G.$ Then $\lambda^2$ is a nonnegative integer, and $G$ is the union of some isolated vertices (if any) and a complete bipartite graph with $\lambda^2$ edges.
\end{prop}
\begin{proof}
By Proposition~\ref{prop_bipartite} and Proposition~\ref{prop_order_size}(ii), $G$ is bipartite with $\lambda^2$ edges. Since the equality meets in Lemma~\ref{lem_sqrte}, the completeness follows.
\end{proof}

\medskip

From Proposition~\ref{prop_complete_bi} one can quickly find the all complete bipartite graphs which are DS.
\begin{cor}
For any positive integers $p\leq q,$ $K_{p,q}$ is DS if and only if $p'\leq p$ and $q'\geq q$ for any positive integers $p'\leq q'$ satisfying $p'q'=pq.$
\end{cor}

\medskip

It is not difficult to compute the spectrum of each bipartite graph $K_{p,q}^-$ or $K_{p,q}^+$ \cite{cfksw:10,lw:15,cfw:15}.
\begin{prop}\label{prop_sp}
Let $2 \leq \min\{p,q\}$ be positive integers.
\begin{itemize}
\item[(i)] The graph $K_{p,q}^-$ has spectrum
\begin{equation}
\left\{
[0]^{p+q-4},\pm\sqrt{\frac{pq-1\pm\sqrt{(pq-1)^2-4(p-1)(q-1)}}{2}}
\right\},~~~\text{and}
\nonumber
\end{equation}
\item[(ii)] the graph $K_{p,q}^+$ has spectrum
\begin{equation}
\left\{
[0]^{p+q-3},\pm\sqrt{\frac{pq+1\pm\sqrt{(pq+1)^2-4(p-1)q}}{2}}
\right\}.
\nonumber
\end{equation}
\end{itemize}
\end{prop}

\medskip

We introduce some sets of bipartite graphs for later use.
\begin{defn}\label{defn_sets_K}
Let the sets of bipartite graphs
\begin{eqnarray*}
\mathbb{K}^0&:=&\{K_{p,q} \mid p,q \in \mathbb{N}\},
\\
\mathbb{K}^-&:=&\{K_{p,q}^- \mid 2 \leq p \leq q, (p,q)\neq(2,2)\},
\\
\mathbb{K}^+&:=&\{K_{p,q}^+ \mid 2 \leq p \leq q\}, \text{ and}
\\
\mathbb{K}&:=&\mathbb{K}^0\cup \mathbb{K}^- \cup \mathbb{K}^+.
\end{eqnarray*}
\end{defn}
Then a Lemma is direct from \cite[Lemma~4.1]{cfw:15}.
\begin{lem}\label{lemma_cheng}
Let $G$ be a simple bipartite graph on $e$ edges without isolated vertices. If the spectral radius $\rho(G)$ of $G$ satisfies
$$\rho(G) \geq \sqrt{\frac{e+\sqrt{e^2-4(e-1-\sqrt{e-1})}}{2}},$$
then $G \in \mathbb{K}.$
\end{lem}

\medskip

The following result \cite[Proposition~1]{dh:03} is well-known.
\begin{prop}\label{prop_path}
The path with $n$ vertices is DS.
\end{prop}

\bigskip

\section{Spectral characterizations of $K_{p,q}^-$}   \label{sec_Kpq-}

Note that the set $\mathbb{K}^-$ of nearly bipartite graphs is defined in Definition~\ref{defn_sets_K}. We prove that each graph $G\in \{K_{2,2}^-\}\cup \mathbb{K}^-$ is DS in this section.

\begin{thm}\label{thm_Kpq-}
For any positive integers $2 \leq p \leq q,$ the graph $K_{p,q}^-$ is DS.
\end{thm}
\begin{proof}
If $p=q=2$ then $K_{p,q}^-$ is a path on $4$ vertices. Hence $K_{2,2}^-$ is DS by Proposition~\ref{prop_path}.
Let $2<q$ and $G$ be a simple graph with the same spectrum as $K_{p,q}^-.$
From Proposition~\ref{prop_order_size}, the numbers of vertices and edges in $G$ are
$|V(G)| = p+q$ and $|E(G)| = pq-1.$
Additionally, Proposition~\ref{prop_bipartite} tells that $G$ is a bipartite graph.

\medskip

Suppose $G$ has at least $2$ nontrivial components $G_1$ and $G_2,$ where a nontrivial component is a connected graph with at least one edge. Then the spectra of $G_1$ and $G_2$ share the nonzero eigenvalues of $G.$ Since $G$ is bipartite, $G_1$ and $G_2$ are both bipartite. By Proposition~\ref{prop_bipartite} again and without loss of generality we have
$\text{sp}(G_1) =\{[0]^{m_1}, \pm e_1\}$
and
$\text{sp}(G_2) = \{[0]^{m_2}, \pm e_2\}$
for some nonnegative integers $m_1,m_2$ with $m_1+m_2+4 \leq p+q,$
where
$$e_1=\sqrt{\frac{pq-1+\sqrt{(pq-1)^2-4(p-1)(q-1)}}{2}}$$
and
$$e_2=\sqrt{\frac{pq-1-\sqrt{(pq-1)^2-4(p-1)(q-1)}}{2}}$$
by  Proposition~\ref{prop_sp}(i).
From Proposition~\ref{prop_complete_bi} $G_1$ is a complete bipartite graph with $e_1$ edges, and thus $(pq-1)^2-4(p-1)(q-1)$ is a perfect square of type $(pq-1-2k)^2$ for some $k \in \mathbb{N}.$ However,
\begin{eqnarray}
(pq-1)^2-4(p-1)(q-1) &=& (pq-1)^2 - 4(pq-p-q+1)
\nonumber \\
&>& (pq-1)^2 - 4(pq-2)
\nonumber \\
&=& (pq-3)^2,
\nonumber
\end{eqnarray}
which is a contradiction.

\medskip

Therefore $G$ has exactly one nontrivial component $G_0.$ Then
\begin{equation}
\text{sp}(G_0)=
\left\{[0]^m,\pm\sqrt{\frac{pq-1\pm\sqrt{(pq-1)^2-4(p-1)(q-1)}}{2}}\right\}
\nonumber
\end{equation}
for some nonnegative integer $m$ with
\begin{equation}\label{eq_num_ver_G0}
|V(G_0)| = m+4 \leq p+q.
\end{equation}
Then by Proposition~\ref{prop_order_size}~(ii)
\begin{equation}\label{eq_num_edge_G0}
e:=|E(G_0)|=|E(G)|=pq-1.
\end{equation}
Note that the spectral radius of $G_0$
\begin{eqnarray}
\rho(G_0)&=&\sqrt{\frac{pq-1+\sqrt{(pq-1)^2-4(p-1)(q-1)}}{2}}
\nonumber \\
&=& \sqrt{\frac{e+\sqrt{e^2-4(e-1-(p+q-3))}}{2}}
\nonumber \\
&\geq& \sqrt{\frac{e+\sqrt{e^2-4(e-1-\sqrt{e-1})}}{2}},
\nonumber
\end{eqnarray}
since
\begin{equation}
(p+q-3)^2-(e-1)=(p-2)^2+(q-3)^2+p(q-2)-2 \geq 0
\nonumber
\end{equation}
for $2\leq p \leq q$ and $3 \leq q.$ By Lemma~\ref{lemma_cheng} $G_0 \in \mathbb{K}$. From Proposition~\ref{prop_complete_bi} $G_0$ is not a complete bipartite graph. Hence $G_0 \in \mathbb{K}^-$ or $G_0 \in \mathbb{K}^+.$ Suppose $G_0 \in \mathbb{K}^+,$ \emph{i.e.}, $G_0 = K_{p',q'}^+$ for some $2 \leq p' \leq q'.$ Then by \eqref{eq_num_ver_G0}
\begin{equation}\label{eq_num_ver_G0_2}
|V(G_0)| = p'+q'+1 \leq p+q,
\end{equation}
and by \eqref{eq_num_edge_G0}
\begin{equation}\label{eq_num_edge_G0_2}
|E(G_0)| = p'q'+1 = e = pq-1.
\end{equation}
According to Proposition~\ref{prop_sp}~(ii),
\begin{equation}\label{eq_p'q'_G0}
(p'-1)q' = (p-1)(q-1).
\end{equation}
\eqref{eq_num_edge_G0_2} and \eqref{eq_p'q'_G0} imply $q'+3=p+q.$ Then by \eqref{eq_num_ver_G0_2} $p'\leq 2,$ and hence $p'=2.$ Therefore $G_0=K_{2,q'}^+=K_{2,q'+1}^-$ for some $q' \geq 2,$ and we have $G_0 \in \mathbb{K}^-.$ Let $G_0=K_{p'',q''}^-$ for some $2 \leq p''\leq q''$ and $3 \leq q''.$ Then we respectively rewrite the equations \eqref{eq_num_ver_G0_2}, \eqref{eq_num_edge_G0_2} and \eqref{eq_p'q'_G0} as
\begin{eqnarray}
|V(G_0)| &=& p''+q'' \leq p+q,
\label{eq_num_ver_G0_3} \\
|E(G_0)| &=& p''q''-1 = pq-1~~~\text{ and}
\label{eq_num_edge_G0_3} \\
(p''-1)(q''-1) &=& (p-1)(q-1),
\label{eq_p''q''_G0}
\end{eqnarray}
where the third equation \eqref{eq_p''q''_G0} is from Proposition~\ref{prop_sp}~(i). \eqref{eq_num_edge_G0_3} and \eqref{eq_p''q''_G0} imply $|V(G_0)|=p''+q''=p+q=|V(G)|.$ Hence $G_0=G.$ The equalities in both sum and product of $p''\leq q''$ and $p \leq q$ imply that $(p'',q'')=(p,q).$ Hence $G = G_0 = K_{p,q}^-,$ and the result follows.
\end{proof}

\begin{rem}\label{rem_Kpq-}
From Theorem~\ref{thm_Kpq-} we have $K_{2,q}^+=K_{2,q+1}^-$ is DS for $2 \leq q.$ However, not all graphs in $\mathbb{K}^+$ are DS. For example, the non-isomorphic graphs
$$K_{m+2,4m+2}^+~~~\text{ and }~~~K_{2m+2,2m+3}^-\cup mK_1$$
are cospectral for each $m \in \mathbb{N}.$
\end{rem}

\begin{cor}
$K_{2,q}^+$ is DS for $2 \leq q.$
\qed
\end{cor}

\bigskip

\section{Spectral characterizations of $K_{p,q}^+$}\label{sec_Kpq+}

Theorem~\ref{thm_Kpq-} shows that for $2 \leq p \leq q$ the graph $K_{p,q}^-$ is DS and its proof seems useful for the study on the graph $K_{p,q}^+.$ However, Remark~\ref{rem_Kpq-} immediately gives a family of $K_{p,q}^+$'s which are not DS. In this section we present a sufficient and necessary condition to determine whether $K_{p,q}^+$ is DS or not for each pair $(p,q).$ Furthermore, we also find the all non-isomorphic cospectral graphs of every $K_{p,q}^+$ which is not DS.

\begin{thm}\label{thm_Kpq+}
Let $3 \leq p \leq q$ be positive integers. Then $K_{p,q}^+$ is not DS if and only if the quadratic polynomial
\begin{equation}\label{quadratic_polynomial}
x^2-(q+3)x+(pq+2)=0
\end{equation}
has two integral roots  $p''\leq q''$ in $[2, \infty).$ Moreover, if $K_{p,q}^+$ is not DS then $K_{p'',q''}^- \cup (p-2)K_1$ is its unique non-isomorphic cospectral graph.
\end{thm}

\begin{proof}
Let $G$ be a simple graph with the same spectrum as $K_{p,q}^+.$
We prove Theorem~\ref{thm_Kpq+} by two steps. In the first part of proof we show that $G_0 \in \mathbb{K}^- \cup \mathbb{K}^+$ where $G_0$ is obtained from $G$ by deleting all the isolated vertices (if any). This process is similar to what we have done in the proof of Theorem~\ref{thm_Kpq-}. In the second part of proof, we prove that $K_{p,q}^+$ is not DS if and only if \eqref{quadratic_polynomial} has two integral roots $p'',q''$ and $G_0 = K_{p'',q''}^-.$ Moreover, if $K_{p,q}^+$ is not DS then its only non-isomorphic cospectral graphs is obtained by adding a corresponding number $p-2$ of isolated vertices to $K_{p'',q''}^-.$

\medskip

From Proposition~\ref{prop_order_size}, the numbers of vertices and edges in $G$ are
$|V(G)| = p+q+1$ and  $|E(G)| = pq+1.$ Additionally, Proposition~\ref{prop_bipartite} tells that $G$ is a bipartite graph. Suppose $G$ has at least $2$ nontrivial components $G_1$ and $G_2$. Then the spectra of $G_1$ and $G_2$ share the nonzero eigenvalues of $G.$ Since $G$ is bipartite, $G_1$ and $G_2$ are both bipartite. By Proposition~\ref{prop_bipartite} and without loss of generality, we have
$sp(G_1) = \{[0]^{m_1},\pm e_1$
and
$sp(G_2) = \{[0]^{m_2},\pm e_2\}$
for some nonnegative integers $m_1,m_2$ with $m_1+m_2+4 \leq p+q+1,$
where $$e_1=\sqrt{\frac{pq+1+\sqrt{(pq+1)^2-4(p-1)q}}{2}}$$
and $$e_2=\sqrt{\frac{pq+1-\sqrt{(pq+1)^2-4(p-1)q}}{2}}$$ by  Proposition~\ref{prop_sp}(ii).
From Proposition~\ref{prop_complete_bi}, $G_1$ is a complete bipartite graph with $e_1$ edges, and thus $(pq+1)^2-4(p-1)q$ is a perfect square of type $(pq+1-2k)^2$ for some $k \in \mathbb{N}.$ However,
\begin{equation}
(pq+1)^2-4(p-1)q > (pq+1)^2 - 4pq = (pq-1)^2,
\nonumber
\end{equation}
which is a contradiction.

\medskip

Therefore $G$ has exactly one nontrivial component $G_0.$ Then
\begin{equation}
\text{sp}(G_0)=
\left\{[0]^m,\pm\sqrt{\frac{pq+1\pm\sqrt{(pq+1)^2-4(p-1)q}}{2}}\right\}
\nonumber
\end{equation}
for some nonnegative integer $m$ with
\begin{equation}\label{eq_num_ver_G0_+}
|V(G_0)| = m+4 \leq p+q+1.
\end{equation}
Then by Proposition~\ref{prop_order_size}~(ii)
\begin{equation}\label{eq_num_edge_G0_+}
e:=|E(G_0)|=|E(G)|=pq+1.
\end{equation}
Note that the spectral radius of $G_0$
\begin{eqnarray}
\rho(G_0)&=&\sqrt{\frac{pq+1+\sqrt{(pq+1)^2-4(p-1)q}}{2}}
\nonumber \\
&=& \sqrt{\frac{e+\sqrt{e^2-4(e-1-q)}}{2}}
\nonumber \\
&\geq& \sqrt{\frac{e+\sqrt{e^2-4(e-1-\sqrt{e-1})}}{2}},
\nonumber
\end{eqnarray}
since
\begin{equation}
q^2-(e-1)=q^2-pq=q(q-p) \geq 0
\nonumber
\end{equation}
for $3 \leq p \leq q.$ By Lemma~\ref{lemma_cheng} $G_0 \in \mathbb{K}$. From Proposition~\ref{prop_complete_bi}, $G_0$ is not a complete bipartite graph. Hence $G_0 \in \mathbb{K}^-$ or $G_0 \in \mathbb{K}^+.$ Here we complete the first part of proof.

\medskip

Suppose $G_0 \in \mathbb{K}^+,$ \emph{i.e.}, $G_0 = K_{p',q'}^+$ for some $3 \leq p' \leq q'.$ Then by \eqref{eq_num_ver_G0_+}
\begin{equation}\label{eq_num_ver_G0_2_+}
|V(G_0)| = p'+q'+1 \leq p+q+1,
\end{equation}
and by \eqref{eq_num_edge_G0_+}
\begin{equation}\label{eq_num_edge_G0_2_+}
|E(G_0)| = p'q'+1 = e = pq+1.
\end{equation}
According to Proposition~\ref{prop_sp}~(ii),
\begin{equation}\label{eq_p'q'_G0_+}
(p'-1)q' = (p-1)q.
\end{equation}
\eqref{eq_num_edge_G0_2_+} and \eqref{eq_p'q'_G0_+} imply $q'=q$ and $p'=p.$ Therefore $G=G_0=K_{p,q}^+.$
\medskip

Suppose $G_0 \in \mathbb{K}^-.$ Let $G_0=K_{p'',q''}^-$ for some $2 \leq p''\leq q''$ and $3 \leq q''.$
Similar to the equations \eqref{eq_num_ver_G0_2_+}, \eqref{eq_num_edge_G0_2_+} and \eqref{eq_p'q'_G0_+} above, $G$ is not DS if and only if there exists integral pair $(p'',q'')$ that satisfies
\begin{eqnarray}
|V(G_0)| &=& p''+q'' \leq p+q+1,
\label{eq_num_ver_G0_3_+} \\
|E(G_0)| &=& p''q''-1 = pq+1~~~\text{ and}
\label{eq_num_edge_G0_3_+} \\
(p''-1)(q''-1) &=& (p-1)q,
\label{eq_p''q''_G0_+}
\end{eqnarray}
where the third equation \eqref{eq_p''q''_G0_+} is from Proposition~\ref{prop_sp}. Note that \eqref{eq_num_edge_G0_3_+} and \eqref{eq_p''q''_G0_+} imply
\begin{equation}
p''+q''=q+3
\label{eq_p''q''_q}
\end{equation}
and hence \eqref{eq_num_ver_G0_3_+} automatically holds.
Conversely, \eqref{eq_num_edge_G0_3_+} and \eqref{eq_p''q''_q} imply
\eqref{eq_p''q''_G0_+}.
 Hence the graph $G_0=K_{p'',q''}^-$ exists if and only if quadratic polynomial in \eqref{quadratic_polynomial} has two integral roots $p'',q''.$
Note that $G_0=K_{p'',q''}^-$ is the only graph found in $\mathbb{K}^-\cup\mathbb{K}^+$ except for $K_{p,q}^+.$
Hence we conclude that for each pair of positive integers $3 \leq p \leq q,$ $K_{p,q}^+$ is not DS if and only if \eqref{quadratic_polynomial} has two integral roots $p'',q''$ and the only non-isomorphic cospectral graph is obtained from $K_{p'',q''}^-$ by adding a corresponding number $(p+q+1)-(p''+q'')=p-2$ of isolated vertices by \eqref{eq_p''q''_q}. Here we complete the second part of proof, and the result follows.
\end{proof}
\medskip

The following lemma helps us to exhaustedly enumerate $K_{p, q}^+$ which has a non-isomorphic cospectral graph by using Theorem~\ref{thm_Kpq+}.
\begin{lem}\label{lem_abb't}
Let $3 \leq p \leq q$  be integers. Then the quadratic polynomial in \eqref{quadratic_polynomial}
has two integral roots $p'', q''\in [2, \infty)$ if and only if
there exist nonnegative integers $a,b,b',t$ with $1 \leq b,b' < a,$ $\text{gcd}(a,b)=1,$ $bb' \equiv 1~(\text{mod}~a),$ and $bb'+t \geq 2$ such that $p,q,p'',q''$ can be written as
\begin{eqnarray}
p&=&b(a-b)t+bb'+\frac{b(1-bb')}{a}+1,
\label{eq_abb't_p} \\
q&=&a^2t+ab',
\label{eq_abb't_q} \\
p''&=&bq/a+1,~~~\text{and}~~~q''=q+2-bq/a.
\label{eq_abb't_p''q''}
\end{eqnarray}
\end{lem}
\begin{proof}
For the necessity, suppose that the quadratic polynomial in \eqref{quadratic_polynomial} has two integral roots $p'', q''\in [2, \infty).$ Then $p''q''=pq+2$ and $p''+q''=q+3.$ Thus
\begin{equation}
(p''-1)\cdot(q+1-(p''-1))=(p-1)q.
\label{eq_multi}
\end{equation}
Note that $p''-1 \leq q,$ otherwise we have $p=2$ which contradicts to $3 \leq p.$ Let
\begin{equation}\label{eq_p''-1}
p''-1=\frac{bq}{a}
\end{equation}
where $1 \leq b <a$ are integers and $\text{gcd}(a,b)=1.$ Then $q \equiv 0$ (mod $a$). Let $q=as$ for some $s \in \mathbb{N}.$
Thus \eqref{eq_multi} becomes
\begin{equation}
\frac{b((a-b)s+1)}{a}=p-1.
\label{eq_multi_2}
\end{equation}
Then $(a-b)s+1 \equiv 0$ (mod $a$) since $\text{gcd}(a,b)=1.$ Hence $bs \equiv 1$ (mod $a$). Let $s=at+b'$ where $t$ and $1\leq b'<a$ are nonnegative integers with $bb'=bs-bat\equiv 1$ (mod $a$). Therefore,
$q=as=a^2t+ab'$ as stated in (\ref{eq_abb't_q}).
Substituting the above $s=at+b'$ into \eqref{eq_multi_2}, we have (\ref{eq_abb't_p}).
The expression formulae of $p''$ and $q''$ are immediate from \eqref{eq_p''-1} and \eqref{eq_p''q''_q}.
Note that if $t=0$ and $bb'=1$ then $p=2$, violating the assumption $p\geq 3.$ Hence $bb'+t\geq 2.$
\medskip

For the sufficiency, we check that for nonnegative integers $a,b,b',t$ satisfying $1 \leq b,b' < a,$ $\text{gcd}(a,b)=1,$ $bb' \equiv 1~(\text{mod}~a),$ and $bb'+t \geq 2,$ the corresponding values of $p,q,p'',q''$ are feasible. Note that we can rewrite \eqref{eq_abb't_q} to \eqref{eq_abb't_p''q''} as
\begin{eqnarray}
p&=&\frac{b[(a-b)(at+b')+1]}{a}+1,
\nonumber \\
q&=&a(at+b'),
\nonumber \\
p''&=&b(at+b')+1,~~~\text{and}~~~q''=(at+b')(a-b)+2.
\nonumber
\end{eqnarray}
One can immediately see that $p'',q''$ are both integers not less than $2.$ Moreover, the sum and product of $p'',q''$ are
$$\left\{
\begin{array}{c}
  p''+q''=q+3 \\
  p''q''=pq+2
\end{array}
\right.,$$
which imply that $p'',q''$ are the two integral roots of \eqref{quadratic_polynomial}.
\end{proof}

\medskip

To quickly find a non-isomorphic cospectral graphs pair which are nearly complete bipartite, a special case of Theorem~\ref{thm_Kpq+} is provided in the following corollary.
\begin{cor}\label{cor_kpq+}
For each pair of positive integers $(t,a)$ with $a \geq 2,$ the graph $$K_{(a-1)t+2,a^2t+a}^+$$ is not DS. Moreover, $$K_{at+2,a(a-1)t+a+1}^- \cup (a-1)tK_1$$ is its unique cospectral graph.
\end{cor}
\begin{proof}
Let $b=b'=1$ in Lemma~\ref{lem_abb't}. Then \eqref{eq_abb't_p} to \eqref{eq_abb't_p''q''} become that
\begin{eqnarray}
p&=&(a-1)t+2,
\nonumber \\
q&=&a^2t+a,
\nonumber \\
p''&=&at+2,~~~\text{and}~~~q''=a(a-1)t+a+1.
\nonumber
\end{eqnarray}
Substituting these data into Theorem~\ref{thm_Kpq+},  we immediately have the proof.
\end{proof}

\medskip

\begin{exam}
By computational programming, we list all $K_{p,q}^+$'s that are not DS for $q \leq 20$ in the following table including the corresponding unique cospectral graphs and the values of parameters $a,b,b',t.$ Note that the choices of $(a,b,b',t)$ are not unique.
\begin{center}
\begin{tabular}{|c|c|c|c|c|c|}
  \hline
   & \text{The unique cospectral graph} & $(a,b,b',t)$ \\
  \hline
  $K_{3,6}^+$ & $K_{4,5}^- \cup K_1$ & $(2,1,1,1)$ or $(3,2,2,0)$ \\
  \hline
  $K_{4,10}^+$ & $K_{6,7}^- \cup 2K_1$ & $(2,1,1,2)$ or $(5,3,2,0)$ \\
  \hline
  $K_{5,14}^+$ & $K_{8,9}^- \cup 3K_1$ & $(2,1,1,3)$ or $(7,4,2,0)$ \\
  \hline
  $K_{4,12}^+$ & $K_{5,10}^- \cup 2K_1$ & $(3,1,1,1)$ or $(4,3,3,0)$ \\
  \hline
  $K_{5,15}^+$ & $K_{7,11}^- \cup 3K_1$ & $(3,2,2,1)$ or $(5,2,3,0)$ \\
  \hline
  $K_{6,18}^+$ & $K_{10,11}^- \cup 4K_1$ & $(2,1,1,4)$ or $(9,5,2,0)$ \\
  \hline
  $K_{5,20}^+$ & $K_{6,17}^- \cup 3K_1$ & $(4,1,1,1)$ or $(5,4,4,0)$ \\
  \hline
\end{tabular}
\end{center}

\end{exam}


\section*{Acknowledgments}
This research is supported by the National Science Council of Taiwan R.O.C. and Ministry of
Science and Technology of Taiwan R.O.C. respectively under the projects NSC 102-2115-M-009-
009-MY3 and MOST 103-2632-M-214-001-MY3-2.

\end{document}